\documentclass[12pt]{amsart}
\usepackage{amssymb,anysize}

\usepackage[backend=bibtex,doi=false,url=false,isbn=false,maxnames=6,giveninits=true]{biblatex}
\renewbibmacro*{journal+issuetitle}
{
  \usebibmacro{journal}
  \setunit*{\addspace}
  \iffieldundef{series}
    {}
    {\newunit
     \printfield{series}
     \setunit{\addspace}}
  \printfield{volume}
  \setunit{\addspace}
  \usebibmacro{issue+date}
  \setunit{\addcolon\space}
  \usebibmacro{issue}
  \setunit{\addcomma\space}
  \printfield{number}
  \newunit
}

\renewbibmacro{in:}{}
\DeclareFieldFormat[article]{title}{\textit{#1}
\addcomma}
\DeclareFieldFormat[article]{journaltitle}{#1}
\DeclareFieldFormat{pages}{#1}
\DeclareFieldFormat[article]{volume}{\mkbibbold{#1}}
\DeclareFieldFormat[article]{number}{no. #1}

\DeclareFieldFormat[book]{title}{\textit{#1}
\addcomma}

\usepackage{subfigure}
\usepackage{tikz}
\usetikzlibrary{knots}
\usetikzlibrary{decorations.markings}

\RequirePackage[colorlinks = true, urlcolor = blue, citecolor = blue, linkcolor = blue,]{hyperref}

\newtheorem{thm}{Theorem}[section]
\newtheorem{lemma}[thm]{Lemma}

\newtheorem{ques}[thm]{Question}

\theoremstyle{definition}
\newtheorem{defn}[thm]{Definition}

\newtheorem{ex}[thm]{Example}

\theoremstyle{remark}

\newcommand{\Z}{\mathbb{Z}}
\newcommand{\R}{\mathbb{R}}
\newcommand{\FT}{\mathrm{FT}}
\newcommand{\HT}{\mathrm{HT}}

\newcommand{\id}{\mathbf{1}}

\title{A full-twisting formula for the HOMFLY polynomial}
\author{Keita Nakagane}
\address{Department of Mathematics,
Tokyo Institute of Technology}
\email{nakagane.k.aa@m.titech.ac.jp}

\begin{document}

\begin{abstract}
We introduce a certain class of link diagrams, which 
includes all closed braid diagrams.
We show a generalized version of K\'alm\'an's full-twist formula for the HOMFLY polynomial in the class.
\end{abstract}

\maketitle

\section{Introduction}
The HOMFLY polynomial \cite{HOMFLY, Jones, PT} $P(L)(v, z)$ is an oriented link invariant, which is defined by the skein relation
\[
    v^{-1}P(
    \begin{tikzpicture}[baseline=1]
        \draw[->] (0,0)--(0.3,0.3);
        \draw (0.3,0)--(0.2,0.1);
        \draw[->] (0.1,0.2)--(0,0.3);
    \end{tikzpicture}
    )-v P(
    \begin{tikzpicture}[baseline=1]
        \draw[->] (0.3,0)--(0,0.3);
        \draw (0,0)--(0.1,0.1);
        \draw[->] (0.2,0.2)--(0.3,0.3);
    \end{tikzpicture}
    )=zP(
    \begin{tikzpicture}[baseline=1]
        \draw [bend right = 45,->] (0,0) to (0,0.3);
        \draw [bend left = 45,->] (0.3,0) to (0.3,0.3);
    \end{tikzpicture}
    ),
\]
and the normalization $P(\text{trivial knot}) = 1$.

The framed HOMFLY polynomial is a useful alternative to study the HOMFLY polynomial.
With the blackboard-framing convention, it is defined for oriented link diagrams $D$, and it is denoted by $H(D)(v, z)$. 
The two versions are related by the formula $H(D) = v^{-w(D)}P(D)$, where $w(D)$ is the writhe of $D$.

We can see that $H(D)$ is invariant under the Reidemeister moves II and III.
The skein relation and the Reidemeister move I relations for $H(D)$ are as follows:
\begin{align*}
    H(
    \begin{tikzpicture}[baseline=1]
        \draw[->] (0,0)--(0.3,0.3);
        \draw (0.3,0)--(0.2,0.1);
        \draw[->] (0.1,0.2)--(0,0.3);
    \end{tikzpicture}
    )-H(&
    \begin{tikzpicture}[baseline=1]
        \draw[->] (0.3,0)--(0,0.3);
        \draw (0,0)--(0.1,0.1);
        \draw[->] (0.2,0.2)--(0.3,0.3);
    \end{tikzpicture}
    )=zH(
    \begin{tikzpicture}[baseline=1]
        \draw [bend right = 45,->] (0,0) to (0,0.3);
        \draw [bend left = 45,->] (0.3,0) to (0.3,0.3);
    \end{tikzpicture}
    ),\\
    H(
    \begin{tikzpicture}[baseline=1]
        \draw (0.15,0.15) to [out=-135,in=45] (0,0);
        \draw (0.3,0.25) to [out=180,in=45] (0.15,0.15);
        \draw (0.3,0.05) to [out=0,in=0] (0.3,0.25);
        \draw (0.2,0.1) to [out=-45,in=180] (0.3,0.05);
        \draw (0,0.3) to [out=135,in=-45] (0.1,0.2);
    \end{tikzpicture}
    )&=v^{-1}H(
    \begin{tikzpicture}[baseline=1]
        \draw (0,0) to [out=30,in=-90] (0.2,0.15);
        \draw (0.2,0.15) to [out=90,in=-30] (0,0.3);
    \end{tikzpicture}
    \,
    ),\\
    H(
    \begin{tikzpicture}[baseline=1]
        \draw (0,0) to [out=45,in=-135] (0.1,0.1);
        \draw (0.2,0.2) to [out=45,in=180] (0.3,0.25);
        \draw (0.3,0.25) to [out=0,in=0] (0.3,0.05);
        \draw (0.3,0.05) to [out=180,in=-45] (0.15,0.15);
        \draw (0.15,0.15) to [out=135,in=-45] (0,0.3);
    \end{tikzpicture}
    )&=vH(
    \begin{tikzpicture}[baseline=1]
        \draw (0,0) to [out=30,in=-90] (0.2,0.15);
        \draw (0.2,0.15) to [out=90,in=-30] (0,0.3);
    \end{tikzpicture}
    \,
    ).
\end{align*}

In this paper, we focus on certain extreme parts of $H(D)$.
They are specified by the \textit{Morton--Franks--Williams} (\textit{MFW}) \textit{bounds} \cite{Morton, FW}.

\begin{thm}[\cite{Morton, FW}] 
    For an oriented link diagram D, we have
    \[
        -s(D)+1 \leq \mathrm{mindeg}_{v}\,H(D) \leq \mathrm{maxdeg}_{v}\,H(D) \leq s(D)-1,
    \]
    where $s(D)$ is the number of Seifert circles of $D$.
\end{thm}

For a diagram $D$, we denote the coefficient of $v^{s(D)-1}$ (resp.\ $v^{-s(D)+1}$) in $H(D)$ by $H_{+}(D)$ (resp.\ $H_{-}(D)$).
Here $H_{+}(D)$ and $H_{-}(D)$ are Laurent polynomials in $z$.
We note that they can be zero when the MFW bounds are not sharp for $D$.

Now we recall the full-twist formula for braids by K\'alm\'an \cite{Kalman}.

\begin{thm}[\cite{Kalman}]\label{braidFT} 
    Let $B$ be an $n$-strand braid.
    Then we have 
    \[
        H_{-}(\widehat{B}) = (-1)^{n-1}H_{+}(\widehat{\FT B}),        
    \]
    where $\FT$ is the positive full-twist braid and the hat $\hat{\cdot}$ of a braid denotes its closure \textup{(}see figure \ref{braidfig}\textup{)}.
\end{thm}

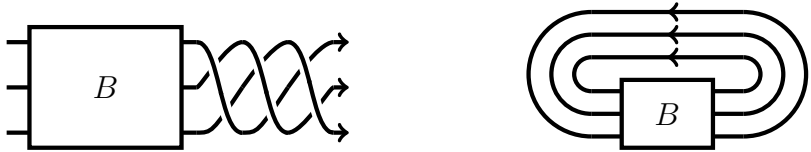
\begin{figure}[h]
\begin{center}
\subfigure{
    \begin{tikzpicture}
        \foreach \y in {0, 0.6, 1.2}{
        \draw[ultra thick, ->] (1.8, \y)--(2.0, \y);
        }
        \begin{knot}
            \strand[ultra thick] (0,1.2) .. controls +(0.3,0) and +(-0.3,0) .. (0.6,0);
            \strand[ultra thick] (0.6,1.2) .. controls +(0.3,0) and +(-0.3,0) .. (1.2,0);
            \strand[ultra thick] (1.2,1.2) .. controls +(0.3,0) and +(-0.3,0) .. (1.8,0);
            \strand[ultra thick] (0,0) .. controls +(0.2,0) .. (0.6,0.6);
            \strand[ultra thick] (0.6,0) .. controls +(0.2,0) .. (1.2,0.6);
            \strand[ultra thick] (0.6,0.6) .. controls +(0,0) and +(-0.2,0) .. (1.2,1.2);
            \strand[ultra thick] (1.2,0.6) .. controls +(0,0) and +(-0.2,0) .. (1.8,1.2);
            \strand[ultra thick] (0,0.6) .. controls +(0.05,0) and +(-0.2,0) .. (0.6,1.2);
            \strand[ultra thick] (1.2,0) .. controls +(0.2,0) and +(-0.05,0) .. (1.8,0.6);
        \end{knot}
        \foreach \y in {0,0.6,1.2}{
        \draw[ultra thick] (-0.2, \y)--(0, \y);
        }
        \draw[ultra thick]  (-2.2,1.4) rectangle (-0.2,-0.2);
        \node at (-1.2,0.6) {$B$};
        \foreach \y in {0,0.6,1.2}{
        \draw[ultra thick] (-2.5, \y)--(-2.2, \y);
        }
    \end{tikzpicture}
}
\hspace{50pt}
\subfigure{
    \begin{tikzpicture}
        \draw[ultra thick]  (-0.6,0.45) rectangle (0.6,-0.45);
        \node at (0,0) {$B$};
        \foreach \y in {-0.3,0,0.3}{
            \draw[ultra thick] (0.6,\y)--(1,\y);
            \draw[ultra thick] (-1,\y)--(-0.6,\y);
        }
        \begin{scope}[decoration={markings, mark=at position 0.5 with {\arrow{>}}}]
            \foreach \y in {1.35,1.05,0.75}{
                \draw[postaction={decorate}, ultra thick] (1,\y)--(-1,\y);
            }
        \end{scope}
        \foreach \y in {-0.3,0,0.3}{
            \draw[ultra thick] (1,\y) arc [start angle = -90, end angle = 90, radius = {0.525-\y}];
            \draw[ultra thick] (-1,\y) arc [start angle = 270, end angle = 90, radius = {0.525-\y}];
        }
    \end{tikzpicture}
}
\end{center}
\caption{$\FT B$ (left) and $\widehat{B}$ (right) for a braid $B$}\label{braidfig}
\end{figure}

This formula is categorified in \cite{Nakagane, EMAN}, that is to say, it is generalized for the Khovanov--Rozansky HOMFLY homology \cite{KR}.
In this paper, we will explore a different type of generalization, namely we will extend the formula to \textit{knitted diagrams}, which include braid closures.

The definitions and the generalized formula are given in section \ref{knitted}.
We prepare some terminology from the Hecke algebra in section \ref{hecke}, then we prove the formula in section \ref{proof}.

\section{knitted diagrams}\label{knitted}
We will define knitted diagrams as a generalization of closed braid diagrams.
They are obtained by arranging several braid diagrams on the plane and connecting them appropriately.
\begin{defn} 
    Let $D \subset \R^2$ be an oriented link diagram.
    An \textit{$n$-strand braid box} in $D$ is a rectangle $B \subset \R^2$ such that the restriction of $D$ in $B$ is an $n$-strand braid diagram.
\end{defn}

We note that 
since all strands cross a braid box in the same direction, 
a Seifert circle of $D$ passes through each braid box in $D$ at most once.

The following abuse of notation will be useful: for a braid box $B$ in $D$, the symbol $B$ will also denote the braid represented by the diagram $D \cap B$.

\begin{defn}\label{pattern}
    Let $D$ be an oriented link diagram.
    A \textit{knitting pattern} of $D$ is a finite set $\{B_i\}_{i=1}^m$ of braid boxes in $D$ such that:
    \begin{itemize}
        \item the braid boxes $B_i$ are disjoint,
        \item every crossing of $D$ is contained in some box $B_i$,
        \item there is no pair of Seifert circles of $D$ both of which pass through two common boxes in $\{B_i\}$.
    \end{itemize}
    A \textit{knitted diagram} is a pair $D = (D, \{B_i\})$ of a diagram $D$ and its knitting pattern $\{B_i\}$.
\end{defn}

Obviously, a closed braid diagram can be considered as a knitted diagram with a single braid box (as in figure \ref{braidfig}).

\begin{ex}
Let $G$ be a simple plane bipartite graph.
By reversing Seifert's algorithm, a special link diagram $D$ can be obtained from $G$, up to crossing signs and orientations, so that the Seifert graph of $D$ is $G$.
By putting a $2$-strand braid box on each crossing of $D$, we obtain a knitted diagram over $D$.
(The simplicity of $G$ is needed for the third condition on knitting patterns in definition \ref{pattern}. See figure \ref{knittedfig} for an example.)
\end{ex}

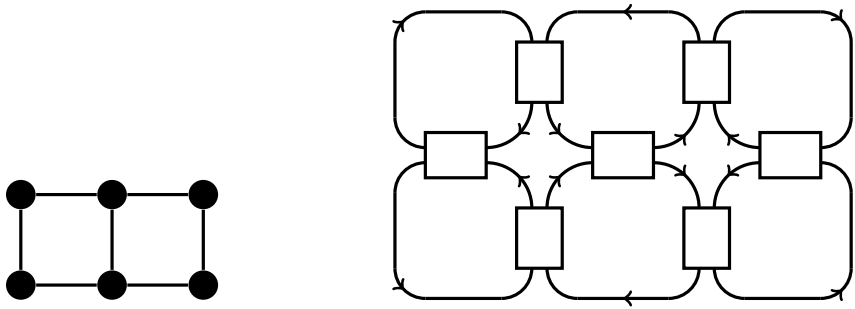
\begin{figure}[h]
\begin{center}
\subfigure{
        \begin{tikzpicture}
            \node[circle, fill] (v2) at (0,0) {};
            \node[circle, fill] (v1) at (0,1.2) {};
            \node[circle, fill] (v4) at (1.2,1.2) {};
            \node[circle, fill] (v3) at (1.2,0) {};
            \node[circle, fill] (v6) at (2.4,1.2) {};
            \node[circle, fill] (v5) at (2.4,0) {};
            \draw[very thick]  (v1) edge (v2);
            \draw[very thick]  (v1) edge (v4);
            \draw[very thick]  (v2) edge (v3);
            \draw[very thick]  (v3) edge (v4);
            \draw[very thick]  (v3) edge (v5);
            \draw[very thick]  (v4) edge (v6);
            \draw[very thick]  (v5) edge (v6);
        \end{tikzpicture}
}
\hspace{50pt}
\subfigure{
    \begin{tikzpicture}
        \draw [very thick] (-1.4,1.5) rectangle (-0.8,0.7);
        \draw [very thick] (-2.6,0.3) rectangle (-1.8,-0.3);
        \draw [very thick] (-1.4,-0.7) rectangle (-0.8,-1.5);
        \draw [very thick] (-0.4,0.3) rectangle (0.4,-0.3);
        \draw [very thick] (0.8,1.5) rectangle (1.4,0.7);
        \draw [very thick] (0.8,-0.7) rectangle (1.4,-1.5);
        \draw [very thick] (1.8,0.3) rectangle (2.6,-0.3);
        \begin{scope}[very thick, decoration={markings, mark=at position 0.5 with {\arrow{>}}}]
            \foreach \t in {1, -1}{
                \draw[postaction={decorate}] (-1.2,0.7*\t) arc [start angle = 0, delta angle = -90*\t, radius = 0.6];
		        \draw[postaction={decorate}] (-1,0.7*\t) arc [start angle = 180, delta angle = 90*\t, radius = 0.6];
		        \draw[postaction={decorate}] (0.4,0.1*\t) arc [start angle = -90*\t, delta angle = 90*\t, radius = 0.6];
                \draw[postaction={decorate}] (1.8,0.1*\t) arc [start angle = -90*\t, delta angle = -90*\t, radius = 0.6];
                \draw (-1.6,1.9*\t) arc [start angle = 90*\t, delta angle = -90*\t, radius = 0.4];
		        \draw (-0.6,1.9*\t) arc [start angle = 90*\t, delta angle = 90*\t, radius = 0.4];
		        \draw (1,1.5*\t) arc [start angle = 0, delta angle = 90*\t, radius = 0.4];
		        \draw (1.2,1.5*\t) arc [start angle = 180, delta angle = -90*\t, radius = 0.4];
		        \draw (-2.6,0.1*\t) arc [start angle = -90*\t, delta angle = -90*\t, radius = 0.4];
		        \draw (3,0.5*\t) arc [start angle = 0, delta angle = -90*\t, radius = 0.4];
		        \draw[postaction={decorate}] (2.6,1.9*\t) arc [start angle = 90*\t, delta angle = -90*\t, radius = 0.4];
		        \draw[postaction={decorate}] (-3,1.5*\t) arc [start angle = 180, delta angle = -90*\t, radius = 0.4];
		        \draw[postaction={decorate}] (0.6,1.9*\t)--(-0.6,1.9*\t);
		        \draw (-2.6,1.9*\t)--(-1.6,1.9*\t);
		        \draw (2.6,1.9*\t)--(1.6,1.9*\t);
		        \draw (-3,0.5*\t)--(-3,1.5*\t);
		        \draw (3,1.5*\t)--(3,0.5*\t);
            }
        \end{scope}
    \end{tikzpicture}
}
\end{center}
\caption{A knitted diagram constructed from a graph}\label{knittedfig}
\end{figure}

We will often replace the braids $B_i$ in a knitted diagram $(D, \{B_i\})$.
For a set of braids $\{\beta_i\}$ to replace $\{B_i\}$, the new knitted diagram is denoted by $D' = (D, \{B_i \to \beta_i\})$.

Now we can state our main result, which is a generalization of theorem \ref{braidFT}.

\begin{thm}\label{mainthm}
    Let $D = (D, \{B_i\}_{i=1}^m)$ be a knitted diagram.
    By adding a positive full-twist to each $B_i$, we obtain a new knitted diagram $\FT D = (D, \{B_i \to \FT B_i\})$.
    Then we have $H_{-}(D) = (-1)^{s(D)-1}H_{+}(\FT D)$, where $s(D)$ is the number of Seifert circles of $D$.
\end{thm}

Figure \ref{thmex} shows a knitted diagram $D$ of a knot and its full-twisted diagram $\FT D$.
The framed HOMFLY polynomial of them are shown as tables in figure \ref{poly}.
In each table, the coefficient of $v^{s}z^{t}$ is displayed at the point $(s,t)$ unless it is zero.
The left table is for $D$, and its bottom-left coefficient is at $(-6,0)$.
The right table is for $\FT D$, and its bottom-right coefficient is at $(6,0)$.
We emphasize that one step between adjacent coefficients in the table is by the vector $(2,0)$ or $(0,2)$.
I.e., we read off from the table that $H(D) = (2+3z^2+z^4)v^{-6}-(1+2z^2+3z^4+z^6)v^{-4}-(1+2z^2+3z^4+z^6)v^{-2}+(2+3z^2+z^4)v^{0}-v^2$.
Since $H_{-}(D) = H_{+}(\FT D) = 2+3z^2+z^4$, theorem \ref{mainthm} is valid for this example.

\begin{figure}[h]
\begin{center}
\subfigure{
    \begin{tikzpicture}
    \begin{scope}[gray, thick]
    \foreach \y in {1,-1}{
    \draw (-0.6, 1.95 * \y) rectangle (0.6, 4.45 * \y);
    \foreach \x in {1, -1}{
    \draw (2.1 * \x, 0.9 * \y) rectangle (1.2 * \x, 0.2 * \y);
    }
    }
    \end{scope}
    \begin{scope}[very thick, decoration={markings, mark=at position 0.5 with {\arrow{>}}}]
        \draw[postaction={decorate}] (0,-1.1) -- (0,1.1);
            
        \foreach \x in {1, -1}{
        \draw (0.6 * \x,0.4) .. controls (0.3 * \x,0.4) and (0.3 * \x,-0.4) .. (0.6 * \x,-0.4);
        \draw (2.7 * \x,0.4) .. controls (3 * \x,0.4) and (3 * \x,-0.4) .. (2.7 * \x,-0.4);
            
        \foreach \y in {1,-1}{
        \draw (0.4 * \x,1.1 * \y) .. controls (0.4 * \x,0.7 * \y) and (0.4 * \x,0.7 * \y) .. (0.6 * \x,0.7 * \y);
        
        \draw (2.7 * \x, 0.7 * \y) arc [start angle = -90 * \y, delta angle = 90 * \x * \y, radius =0.2];
        \draw (2.7 * \x, 5.5 * \y) arc [start angle = 90 * \y, delta angle = -90 * \x * \y, radius =0.2];
        \draw (0.6 * \x, 5.5 * \y) arc [start angle = 90 * \y, delta angle = 90 * \x * \y, radius =0.2];
        \draw (0.6 * \x,5.5 * \y) -- (2.7 * \x,5.5 * \y);

        \draw (0.6 * \x,0.4 * \y) -- (1.3 * \x, 0.4 * \y);
        \draw (2 * \x, 0.4 * \y) -- (2.7 * \x, 0.4 * \y);
        \draw (0.6 * \x,0.7 * \y) -- (1.3 * \x, 0.7 * \y);
        \draw (2 * \x, 0.7 * \y) -- (2.7 * \x, 0.7 * \y);
        }
        \draw[postaction={decorate}] (2.9 * \x,5.3) -- (2.9 * \x,0.9);
        \draw[postaction={decorate}] (2.9 * \x,-0.9) -- (2.9 * \x,-5.3);
        }
        \foreach \y in {1,-1}{
        \draw (0,5.3 * \y) -- (0, 5.7 * \y);
        \draw (0,5.7 * \y) arc [start angle = 0, delta angle = 90 * \y, radius = 0.2];
        \draw (-0.2,5.9 * \y) --(-2.7,5.9 * \y);
        \draw (-2.7,5.9 * \y) arc [start angle = 90 * \y, delta angle = 90 * \y, radius = 0.6];
        }
        \draw[postaction={decorate}] (-3.3, 5.3) -- (-3.3, -5.3);
        \begin{knot}[clip width=2,flip crossing/.list = {2,3,5,6,10,12}]
        \strand (1.3,0.7) .. controls +(0.3,0) and +(-0.3,0) .. (2,0.4);
        \strand (1.3,0.4) .. controls +(0.3,0) and +(-0.3,0) .. (2,0.7);
        \strand (-1.3,0.7) .. controls +(-0.3,0) and +(0.3,0) .. (-2,0.4);
        \strand (-1.3,0.4) .. controls +(-0.3,0) and +(0.3,0) .. (-2,0.7);
        \strand (1.3,-0.7) .. controls +(0.3,0) and +(-0.3,0) .. (2,-0.4);
        \strand (1.3,-0.4) .. controls +(0.3,0) and +(-0.3,0) .. (2,-0.7);
        \strand (-1.3,-0.7) .. controls +(-0.3,0) and +(0.3,0) .. (-2,-0.4);
        \strand (-1.3,-0.4) .. controls +(-0.3,0) and +(0.3,0) .. (-2,-0.7);
        \strand (0.4,2.45) .. controls +(0,0.3) and +(0,-0.3) .. (-0.4,3.35) .. controls +(0,0.2) and +(0,-0.2) .. (0, 3.95);
        \strand (-0.4,2.75) .. controls +(0,0.3) and +(0,-0.3) .. (0.4,3.65);
        \strand (0,2.45) .. controls +(0,0.2) and +(0,-0.2) .. (0.4,3.05) .. controls +(0,0.3) and +(0,-0.3) .. (-0.4,3.95);
        \draw (0.4,1.1) -- (0.4,2.45);
        \draw (0,1.1) -- (0,2.45);
        \draw (-0.4,1.1) -- (-0.4,2.75);
        \draw (0.4,3.65) -- (0.4,5.3);
        \draw (0,3.95) -- (0,5.3);
        \draw (-0.4,3.95) -- (-0.4,5.3);
        \strand (0,-4.25) .. controls +(0,0.2) and +(0,-0.2) .. (0.4,-3.65) .. controls +(0,0.3) and +(0,-0.3) .. (-0.4,-2.75) .. controls +(0,0.2) and +(0,-0.2) .. (0,-2.15);
        \strand (0.4,-4.25) .. controls +(0,0.2) and +(0,-0.2) .. (0,-3.65) .. controls +(0,0.2) and +(0,-0.2) .. (0.4,-3.05);
        \strand (-0.4,-3.35) .. controls +(0,0.2) and +(0,-0.2) .. (0,-2.75) .. controls +(0,0.2) and +(0,-0.2) .. (-0.4, -2.15);
        \draw (0.4,-5.3) -- (0.4,-4.25);
        \draw (0,-5.3) -- (0,-4.25);
        \draw (-0.4,-5.3) -- (-0.4,-3.35);
        \draw (0.4,-3.05) -- (0.4,-1.1);
        \draw (0,-2.15) -- (0,-1.1);
        \draw (-0.4,-2.15) -- (-0.4,-1.1);
        \end{knot}
    \end{scope}
    \end{tikzpicture}
}
\hspace{50pt}
\subfigure{
    \begin{tikzpicture}
        \begin{scope}[gray, thick]
        \foreach \y in {1,-1}{
        \draw (-0.6, 5.3 * \y) rectangle (0.6, 1.1 * \y);
        \foreach \x in {1, -1}{
        \draw (2.7 * \x, 0.9 * \y) rectangle (0.6 * \x, 0.2 * \y);
        }
        }
        \end{scope}
        \begin{scope}[very thick, decoration={markings, mark=at position 0.5 with {\arrow{>}}}]
        \draw[postaction={decorate}] (0,-1.1) -- (0,1.1);
        \foreach \x in {1, -1}{
        \draw (0.6 * \x,0.4) .. controls (0.3 * \x,0.4) and (0.3 * \x,-0.4) .. (0.6 * \x,-0.4);
        \draw (2.7 * \x,0.4) .. controls (3 * \x,0.4) and (3 * \x,-0.4) .. (2.7 * \x,-0.4);

        \foreach \y in {1,-1}{
        \draw (0.4 * \x,1.1 * \y) .. controls (0.4 * \x,0.7 * \y) and (0.4 * \x,0.7 * \y) .. (0.6 * \x,0.7 * \y);
        \draw (2.7 * \x, 0.7 * \y) arc [start angle = -90 * \y, delta angle = 90 * \x * \y, radius =0.2];
        \draw (2.7 * \x, 5.5 * \y) arc [start angle = 90 * \y, delta angle = -90 * \x * \y, radius =0.2];
        \draw (0.6 * \x, 5.5 * \y) arc [start angle = 90 * \y, delta angle = 90 * \x * \y, radius =0.2];
        \draw (0.6 * \x,5.5 * \y) -- (2.7 * \x,5.5 * \y);
        }
        \draw[postaction={decorate}] (2.9 * \x,5.3) -- (2.9 * \x,0.9);
        \draw[postaction={decorate}] (2.9 * \x,-0.9) -- (2.9 * \x,-5.3);
        }
        \foreach \y in {1,-1}{
        \draw (0,5.3 * \y) -- (0, 5.7 * \y);
        \draw (0,5.7 * \y) arc [start angle = 0, delta angle = 90 * \y, radius = 0.2];
        \draw (-0.2,5.9 * \y) --(-2.7,5.9 * \y);
        \draw (-2.7,5.9 * \y) arc [start angle = 90 * \y, delta angle = 90 * \y, radius = 0.6];
        }
        \draw[postaction={decorate}] (-3.3, 5.3) -- (-3.3, -5.3);

        \begin{knot}[clip width=2,flip crossing/.list = {2,4,6,7,9,11,13,14,18,20,21,24,26,28,30,31}]
        \strand (0.6,0.7) .. controls +(0.2,0) and +(-0.2,0) .. (1.3,0.4) .. controls +(0.2,0) and +(-0.2,0) .. (2,0.7) .. controls +(0.2,0) and +(-0.2,0) .. (2.7,0.4);
        \strand (0.6,0.4) .. controls +(0.2,0) and +(-0.2,0) .. (1.3,0.7) .. controls +(0.2,0) and +(-0.2,0) .. (2,0.4) .. controls +(0.2,0) and +(-0.2,0) .. (2.7,0.7);
        \strand (-0.6,0.7) .. controls +(-0.2,0) and +(0.2,0) .. (-1.3,0.4) .. controls +(-0.2,0) and +(0.2,0) .. (-2,0.7) .. controls +(-0.2,0) and +(0.2,0) .. (-2.7,0.4);
        \strand (-0.6,0.4) .. controls +(-0.2,0) and +(0.2,0) .. (-1.3,0.7) .. controls +(-0.2,0) and +(0.2,0) .. (-2,0.4) .. controls +(-0.2,0) and +(0.2,0) .. (-2.7,0.7);
        \strand (0.6,-0.7) .. controls +(0.2,0) and +(-0.2,0) .. (1.3,-0.4) .. controls +(0.2,0) and +(-0.2,0) .. (2,-0.7) .. controls +(0.2,0) and +(-0.2,0) .. (2.7,-0.4);
        \strand (0.6,-0.4) .. controls +(0.2,0) and +(-0.2,0) .. (1.3,-0.7) .. controls +(0.2,0) and +(-0.2,0) .. (2,-0.4) .. controls +(0.2,0) and +(-0.2,0) .. (2.7,-0.7);
        \strand (-0.6,-0.7) .. controls +(-0.2,0) and +(0.2,0) .. (-1.3,-0.4) .. controls +(-0.2,0) and +(0.2,0) .. (-2,-0.7) .. controls +(-0.2,0) and +(0.2,0) .. (-2.7,-0.4);
        \strand (-0.6,-0.4) .. controls +(-0.2,0) and +(0.2,0) .. (-1.3,-0.7) .. controls +(-0.2,0) and +(0.2,0) .. (-2,-0.4) .. controls +(-0.2,0) and +(0.2,0) .. (-2.7,-0.7);
        \strand (0.4,1.4) .. controls +(0,0.3) and +(0,-0.3) .. (-0.4,2.3) .. controls +(0,0.2) and +(0,-0.2) .. (0, 2.9);
        \strand (-0.4,1.7) .. controls +(0,0.3) and +(0,-0.3) .. (0.4,2.6);
        \strand (0,1.4) .. controls +(0,0.2) and +(0,-0.2) .. (0.4,2) .. controls +(0,0.3) and +(0,-0.3) .. (-0.4,2.9);
        \draw (0.4,1.1) -- (0.4, 1.4);
        \draw (0,1.1) -- (0,1.4);
        \draw (-0.4,1.1) -- (-0.4,1.7);
        \draw (0.4,2.6) -- (0.4,2.9);
        \strand (-0.4,2.9) .. controls +(0,0.3) and +(0,-0.3) .. (0.4,3.8) .. controls +(0,0.3) and +(0,-0.3) .. (-0.4, 4.7);
        \strand (0.4,3.2) .. controls +(0,0.3) and +(0,-0.3) .. (-0.4,4.1) .. controls +(0,0.3) and +(0,-0.3) .. (0.4, 5);
        \strand (0,2.9) .. controls +(0,0.2) and +(0,-0.2) .. (-0.4,3.5) .. controls +(0,0.3) and +(0,-0.3) .. (0.4,4.4) .. controls +(0,0.2) and +(0,-0.2) .. (0,5);
        \draw (0.4,5) -- (0.4, 5.3);
        \draw (0,5) -- (0,5.3);
        \draw (-0.4,4.7) -- (-0.4,5.3);
        \draw (0.4,3.2) -- (0.4,2.9);
        \strand (0,-5.3) .. controls +(0,0.2) and +(0,-0.2) .. (0.4,-4.7) .. controls +(0,0.3) and +(0,-0.3) .. (-0.4,-3.8) .. controls +(0,0.2) and +(0,-0.2) .. (0,-3.2);
        \strand (0.4,-5.3) .. controls +(0,0.2) and +(0,-0.2) .. (0,-4.7) .. controls +(0,0.2) and +(0,-0.2) .. (0.4,-4.1);
        \strand (-0.4,-4.4) .. controls +(0,0.2) and +(0,-0.2) .. (0,-3.8) .. controls +(0,0.2) and +(0,-0.2) .. (-0.4, -3.2);
        \draw (0.4,-4.1) -- (0.4,-3.2);
        \draw (-0.4,-5.3) -- (-0.4,-4.4);
        \strand (-0.4,-3.2) .. controls +(0,0.3) and +(0,-0.3) .. (0.4,-2.3) .. controls +(0,0.3) and +(0,-0.3) .. (-0.4, -1.4);
        \strand (0.4,-2.9) .. controls +(0,0.3) and +(0,-0.3) .. (-0.4,-2) .. controls +(0,0.3) and +(0,-0.3) .. (0.4, -1.1);
        \strand (0,-3.2) .. controls +(0,0.2) and +(0,-0.2) .. (-0.4,-2.6) .. controls +(0,0.3) and +(0,-0.3) .. (0.4,-1.7) .. controls +(0,0.2) and +(0,-0.2) .. (0,-1.1);
        \draw (0.4,-3.2) -- (0.4,-2.9);
        \draw (-0.4, -1.4) --(-0.4,-1.1);
        \end{knot}
    \end{scope}
    \end{tikzpicture}
}
\end{center}
\caption{A knitted diagram $D$ and $\FT D$}\label{thmex}
\end{figure}
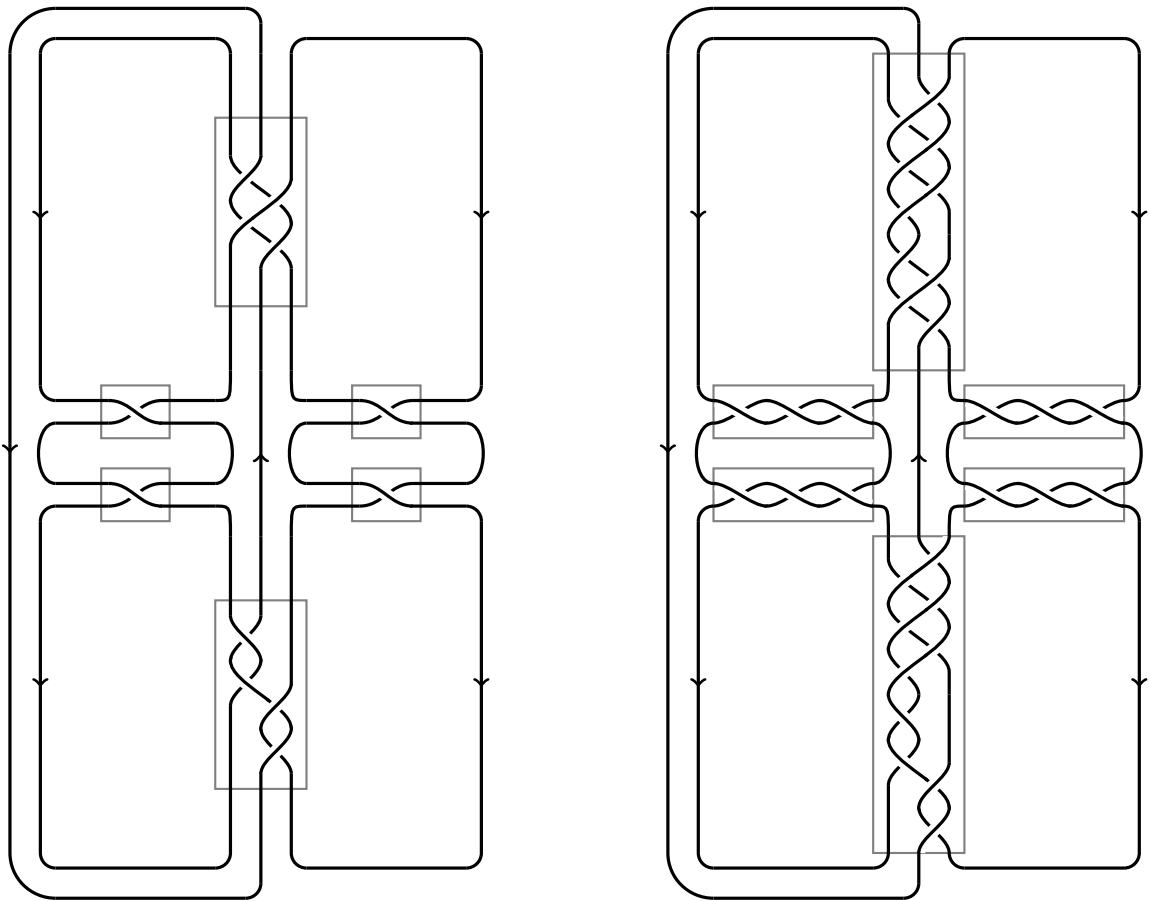

\begin{figure}[b]
\begin{center}
\subfigure{
\begin{tikzpicture}
\node (z) at (1,7) {$z$};
\node (v) at (2,6) {$v$};
\draw[ultra thick, ->] (1, 6) -- (1,6.7);
\draw[ultra thick, ->] (1, 6) -- (1.7,6);
\draw[rounded corners=10pt] (-0.4,-0.4)--(-0.4,2.4)--(0.4,2.4)--(0.4,-0.4)--cycle;
\node at (0,0) {\scriptsize 2};
\node at (1,0) {\scriptsize -1};
\node at (2,0) {\scriptsize -1};
\node at (3,0) {\scriptsize 2};
\node at (4,0) {\scriptsize -1};
\node at (0,1) {\scriptsize 3};
\node at (1,1) {\scriptsize -2};
\node at (2,1) {\scriptsize -2};
\node at (3,1) {\scriptsize 3};
\node at (0,2) {\scriptsize 1};
\node at (1,2) {\scriptsize -3};
\node at (2,2) {\scriptsize -3};
\node at (3,2) {\scriptsize 1};
\node at (1,3) {\scriptsize -1};
\node at (2,3) {\scriptsize -1};

\end{tikzpicture}
}
\hspace{50pt}
\subfigure{
\begin{tikzpicture}
    \draw[rounded corners=10pt] (5.6,-0.4)--(5.6,2.4)--(6.4,2.4)--(6.4,-0.4)--cycle;
    \node at (0,0) {\scriptsize 112};
    \node at (1,0) {\scriptsize -336};
    \node at (2,0) {\scriptsize 419};
    \node at (3,0) {\scriptsize -281};
    \node at (4,0) {\scriptsize 107};
    \node at (5,0) {\scriptsize -22};
    \node at (6,0) {\scriptsize 2};
    \node at (0,1) {\scriptsize 1008};
    \node at (1,1) {\scriptsize -2384};
    \node at (2,1) {\scriptsize 2328};
    \node at (3,1) {\scriptsize -1212};
    \node at (4,1) {\scriptsize 344};
    \node at (5,1) {\scriptsize -49};
    \node at (6,1) {\scriptsize 3};
    \node at (0,2) {\scriptsize 3864};
    \node at (1,2) {\scriptsize -6816};
    \node at (2,2) {\scriptsize 4921};
    \node at (3,2) {\scriptsize -1892};
    \node at (4,2) {\scriptsize 359};
    \node at (5,2) {\scriptsize -27};
    \node at (6,2) {\scriptsize 1};
    \node at (0,3) {\scriptsize 8416};
    \node at (1,3) {\scriptsize -10076};
    \node at (2,3) {\scriptsize 4817};
    \node at (3,3) {\scriptsize -1314};
    \node at (4,3) {\scriptsize 125};
    \node at (5,3) {\scriptsize 4};
    \node at (0,4) {\scriptsize 11655};
    \node at (1,4) {\scriptsize -7747};
    \node at (2,4) {\scriptsize 1727};
    \node at (3,4) {\scriptsize -498};
    \node at (4,4) {\scriptsize 13};
    \node at (5,4) {\scriptsize 6};
    \node at (0,5) {\scriptsize 10833};
    \node at (1,5) {\scriptsize -1804};
    \node at (2,5) {\scriptsize -554};
    \node at (3,5) {\scriptsize -424};
    \node at (4,5) {\scriptsize 36};
    \node at (5,5) {\scriptsize 1};
    \node at (0,6) {\scriptsize 6925};
    \node at (1,6) {\scriptsize 2052};
    \node at (2,6) {\scriptsize -508};
    \node at (3,6) {\scriptsize -509};
    \node at (4,6) {\scriptsize 33};
    \node at (0,7) {\scriptsize 3055};
    \node at (1,7) {\scriptsize 2206};
    \node at (2,7) {\scriptsize 112};
    \node at (3,7) {\scriptsize -317};
    \node at (4,7) {\scriptsize 10};
    \node at (0,8) {\scriptsize 914};
    \node at (1,8) {\scriptsize 1013};
    \node at (2,8) {\scriptsize 216};
    \node at (3,8) {\scriptsize -101};
    \node at (4,8) {\scriptsize 1};
    \node at (0,9) {\scriptsize 177};
    \node at (1,9) {\scriptsize 257};
    \node at (2,9) {\scriptsize 86};
    \node at (3,9) {\scriptsize -16};
    \node at (0,10) {\scriptsize 20};
    \node at (1,10) {\scriptsize 35};
    \node at (2,10) {\scriptsize 15};
    \node at (3,10) {\scriptsize -1};
    \node at (0,11) {\scriptsize 1};
    \node at (1,11) {\scriptsize 2};
    \node at (2,11) {\scriptsize 1};
\end{tikzpicture}
}
\end{center}
\caption{The framed HOMFLY polynomials of $D$ and $\FT D$}\label{poly}
\end{figure}

\section{Hecke algebra}\label{hecke}
We need some terminology from the (type $A$) Hecke algebra to prove our result.

\begin{defn} 
    The Hecke algebra $H_n$ is obtained from the group algebra of the braid group $\mathrm{Br}_n$, over the ring $\Z[z, z^{-1}]$, by imposing the skein relation
    \[
        \sigma_i -\sigma_i^{-1} = z \quad (i = 1,\dots, n-1),
    \]
    where $\sigma_1, \ldots, \sigma_{n-1}$ are the standard generators of $\mathrm{Br}_n$.
\end{defn}

As a $\Z[z, z^{-1}]$-module, the Hecke algebra $H_n$ has several well-known bases.
The positive permutation braids (PPBs) $\{T_w\}_{w \in S_n}$ and the negative permutation braids (NPBs) $\{U_w\}_{w \in S_n}$ are used as bases here.
They are indexed by the symmetric group $S_n$, and the braid $T_w$ (resp.\ $U_w$) is the positive (resp.\ negative) braid representing the permutation $w$ with the fewest crossings possible in its diagram.
We note that $T_{\delta}$ is the positive half-twist braid $\HT$, where here $\delta$ is the longest element in $S_n$:
\[
    \delta = \left(
        \begin{array}{ccccc}
          1 & 2 & \cdots & n-1 & n \\
          n & n-1 & \cdots & 2 & 1
        \end{array}
      \right).
\]

An important property of these bases is that multiplying by $\HT$ from the left (or right) yields a bijection from $\{U_w\}_{w \in S_n}$ to $\{T_w\}_{w \in S_n}$.
The following fact is also essential for the full-twist phenomenon.

\begin{lemma}[{\cite[proposition 3.1]{Kalman}}]\label{Kal}
    Let $x$ be an element in the Hecke algebra $H_n$.
    We expand $x$ in terms of PPBs and NPBs in $H_n$ as follows:
    \[
        x = \sum_{w \in S_n} a_wT_w =  \sum_{w \in S_n} b_wU_w \quad (a_w, b_w \in \Z[z, z^{-1}]).
    \]
    Then we have $a_{\delta} = b_{\delta}$, where $\delta$ is the longest element in $S_n$.
\end{lemma}

For more detail on the type $A$ Hecke algebra, see \cite{Kalman, KT}.

\section{Proof}\label{proof}

We will prove theorem \ref{mainthm} in this section, by generalizing the original proof in \cite{Kalman}.
A key tool here is the refined MFW bound by Murasugi and Przytycki \cite{MP}.

\begin{lemma}[Corollary of {\cite[theorem 8.3]{MP}}]\label{MPcor}
    Let $D$ be an oriented link diagram. If there is a pair of Seifert circles of $D$ with exactly one 
    crossing between them, which is 
    positive \textup{(}resp.\ negative\textup{)}, then we have $H_{+}(D) = 0$ \textup{(}resp.\ $H_{-}(D) = 0$\textup{)}.
\end{lemma}

Indeed, with the condition in the above lemma, the index \cite[definition 2.1]{MP} of the Seifert graph of $D$ becomes positive. It then follows from \cite[theorem 8.3]{MP} that the MFW bound cannot be sharp.

\begin{proof}[Proof of Theorem \ref{mainthm}]
    We denote the framed HOMFLY polynomial of the knitted diagram $(D, \{B_i \to \beta_i\}_{i=1}^m)$ by $H(\beta_i)_{i=1}^m$, for a replacing set $\{\beta_i\}$ of braids.
    
    Let $n_i$ be the number of strands in the braid $B_i$.
    We expand each $B_i$ in terms of PPBs in $H_{n_i}$ and write $B_i = \sum_{w} a_{w}^iT_{w}$,
    where $w$ runs over the symmetric group $S_{n_i}$.
    Note that this expansion only requires the skein relation and braid isotopies.
    Hence, by the same computation, the framed HOMFLY polynomial $H(D)$ can be expanded into a linear sum
    \begin{equation}\label{pos}
    H(D) = \sum_{w_1, \dots, w_m} a_{w_1}^1\cdots a_{w_m}^mH(T_{w_i})_{i=1}^m,
    \end{equation}
    where each $w_i$ runs over $S_{n_i}$.
    (We stress that $w_i$ is \emph{not} determined by $i$, rather it indicates the $i$-th set of subscripts.)
    We expand $H(D)$ also in terms of NPBs, and write
    \begin{equation}\label{neg}
        H(D) = \sum_{w_1, \dots, w_m} b_{w_1}^1\cdots b_{w_m}^mH(U_{w_i})_{i=1}^m.
    \end{equation}
    Let us apply negative half-twists $\HT^{-1}$ to (\ref{pos}) and apply positive half-twists $\HT$ to (\ref{neg}).
    By restricting the equations to $H_{\mp}$, we have 
    \begin{align}
        H_{-}(\HT^{-1}D) &= \sum_{w_1, \dots, w_m} a_{w_1}^1\cdots a_{w_m}^mH_{-}(\HT^{-1}T_{w_i}), \label{hminus}\\
        H_{+}(\HT D) &= \sum_{w_1, \dots, w_m} b_{w_1}^1\cdots b_{w_m}^mH_{+}(\HT U_{w_i}). \label{hplus}
    \end{align}
    We recall that the set $\{\HT^{-1}T_{w_i}\}_{w_i}$ is equal to $\{U_{w_i}\}_{w_i}$.
    Unless $w_i = 1$, a reduced diagram of the braid $U_{w_i}$ has a pair of adjacent strands with exactly one negative crossing between them.
    Hence by the definition of knitting patterns and lemma \ref{MPcor}, we have $H_{-}(U_{w_i}) = 0$ unless $w_i = 1$ for all $i$.
    Therefore, in the right hand side of (\ref{hminus}), only one term can be non-zero; it is given by the longest elements $\delta_i \in S_{n_i}$.
    Now we have
    \[
        H_{-}(\HT^{-1}D) = a_{\delta_1}^1\cdots a_{\delta_m}^mH_{-}(\id_i),
    \]
    where $\id_i$ is the trivial braid $\HT^{-1}T_{\delta_i} = U_1$.
    In a similar way, we have
    \[
        H_{+}(\HT D) = b_{\delta_1}^1\cdots b_{\delta_m}^mH_{+}(\id_i).
    \]
    from (\ref{hplus}).
    Since $(D, \{B_i \to \id_i\})$ is a trivial diagram, we have $H(\id_i) = \{(v^{-1}-v)/z\}^{s(D)-1}$.
    With lemma \ref{Kal}, it follows that
    \begin{align*}
        H_{-}(\HT^{-1}D) &= a_{\delta_1}^1\cdots a_{\delta_m}^m z^{1-s(D)}\\
        &= b_{\delta_1}^1\cdots b_{\delta_m}^m  z^{1-s(D)}\\
        &= (-1)^{s(D)-1}H_{+}(\HT D).
    \end{align*}
    The claimed formula is obtained by replacing $D$ with $\HT D$.
\end{proof}

We end with an open question regarding categorification.
Unfortunately, a direct way to compute the Khovanov--Rozansky HOMFLY homology of a link from an arbitrary diagram is not known; rather, we need a braid representation of the link \cite{KR}.
However, in light of the full-twist formulas of \cite{Nakagane,EMAN} it is still reasonable to ask the following.

\begin{ques}
Can theorem \ref{mainthm} be generalized for the HOMFLY homology?
\end{ques}

\subsection*{Acknowledgements}
I would like to thank my supervisor Tam\'as K\'alm\'an, for valuable suggestions and continuous feedback. This work was supported by JSPS KAKENHI Grant Number JP19J12350.

\printbibliography

\end{document}